\documentclass[12pt]{amsart}
\usepackage[utf8]{inputenc}
\usepackage[english]{babel}
\usepackage{dsfont}
\usepackage{float}
\usepackage[totalheight=24 true cm, totalwidth=17 true cm]{geometry}
\usepackage{enumerate}
\usepackage{bbding}
\usepackage{amsmath,multirow}
\usepackage{amsthm}
\usepackage{amssymb}
\usepackage{mathrsfs}
\usepackage{mathtools}
\usepackage{stmaryrd}
\usepackage{url}
\usepackage{wrapfig}
\usepackage{starfont}
\usepackage{pifont}
\usepackage{eurosym}
\usepackage{subcaption}
\usepackage{csquotes}

\usepackage{multicol}
\usepackage{relsize}
\usepackage{float}
\usepackage{tikz,pgf}
\usepackage{tikz-cd}
\usepackage{wasysym}
\usepackage{graphicx}
\usepackage{fancyhdr}
\usepackage{verbatim}

\usepackage{pgfplots}
\pgfplotsset{compat=1.15}
\usepackage{mathrsfs}
\usetikzlibrary{arrows}

\usepackage[all,knot,arc,color,web]{xy}
\xyoption{arc}
\xyoption{web}
\xyoption{curve}

\newcommand{\kk}{K}
\newcommand{\II}{\mathcal{I}}

\numberwithin{equation}{section}

\theoremstyle{plain}
\newtheorem{theorem}{Theorem}[section]

\newtheorem{corollary}[theorem]{Corollary}
\newtheorem{lemma}[theorem]{Lemma}

\theoremstyle{definition}
\newtheorem{definition}[theorem]{Definition}

\newtheorem{remark}[theorem]{Remark}

\newtheorem{example}[theorem]{Example}
\newtheorem{question}[theorem]{Question}

\usepackage[bookmarksnumbered=true]{hyperref}
\newtheorem{introtheorem}{Theorem}[section]
\hypersetup{
  colorlinks = true,
  linkcolor = blue,
  anchorcolor = blue,
  citecolor = teal,
  filecolor = blue,
  urlcolor = blue
}

\DeclareMathOperator{\rank}{rank}

\DeclareMathOperator{\Span}{span}

\newcommand{\exqed}{\hfill$\triangle$}

\usepackage[maxbibnames=99,style=alphabetic,doi=false,url=false,isbn=false,eprint=false]{biblatex}
\title[Terracini matroids]{Terracini matroids: \\
algebraic matroids of secants and embedded joins}
\author{Fatemeh Mohammadi}
\address{Departments of Computer Science and Mathematics, KU Leuven}
\email{fatemeh.mohammadi@kuleuven.be}

\author{Jessica Sidman}
\address{Department of Mathematics, Amherst College}
\email{jsidman@amherst.edu}

\author{Louis Theran}
\address{School of Mathematics and Statistics, University of St Andrews}
\email{lst6@st-andrews.ac.uk}
\date{}

\begin{document}
\begin{abstract}
  Applications of algebraic geometry have sparked much recent work on
  algebraic matroids. An algebraic matroid encodes algebraic
  dependencies among coordinate functions on a variety.
  We study the behavior of algebraic matroids under joins and secants
  of varieties.
  Motivated by Terracini's lemma, we introduce the notion of a
  Terracini union of matroids, which
  captures when the algebraic matroid of a join coincides with the
  matroid union of the algebraic
  matroids of its summands. We illustrate applications of our results
  with a discussion of the implications for toric surfaces and threefolds.
\end{abstract}

\maketitle
\renewcommand{\thefootnote}{}
\footnotetext{\noindent{\!\!\!\!\!\!\!\small{Keywords:}} algebraic
  matroids, join of varieties, secant varieties, Terracini lemma.
}
\renewcommand{\thefootnote}{\arabic{footnote}}
\section{Introduction}

In many applications of algebraic geometry, the coordinates of the  ambient space of the varieties studied have a meaningful interpretation, and  whether a set of coordinates is independent has important consequences.  The independence of coordinates on a variety $X$ is captured by the structure of an algebraic matroid $M(X)$.  Interestingly, in two applications, algebraic statistics and rigidity theory, secant varieties (and more generally, the embedded join of a collection of varieties) are of great interest (see \cite{allman2009identifiability, garcia2005algebraic, cruickshank2023identifiability}).

Recently, algebraic matroids were studied in work of Cruikshank et al.\ \cite{cruickshank2023identifiability} who proposed the notion of ``$g$-rigidity,'' which
extends ideas from classical rigidity theory to certain unirational varieties.  Inspired by their work, we seek to understand how secant and join constructions correspond to combinatorial operations on their matroids and ask a more fundamental question:

\begin{question}
  What is the relationship between the algebraic matroids of varieties $X_1, \ldots, X_s$ and the algebraic matroid of their join $X_1 + \cdots + X_s$?
\end{question}

We quickly sketch the essential background so we can give intuition for the main results. An irreducible variety $X \subseteq \mathbb{A}^{n}$, with the vanishing ideal
$\II(X) \subseteq \kk[z_1, \ldots, z_n]$, determines an algebraic
matroid $M(X)$ on the ground set $\{z_1, \ldots, z_n\}$ by the rule
that a set of variables $E$ is dependent  if and only if $\II(X)$
contains a non-zero polynomial with support contained in $E$.
Alternatively, the independent sets correspond to coordinate
subspaces onto which $X$ projects dominantly.
This perspective links combinatorics and algebraic geometry: the rank of $M(X)$
equals $\dim X$, a circuit $C$ in $M(X)$ corresponds to 
a polynomial in 
$\II(X)$ that cuts out the vanishing ideal of the projection of $X$ to the 
coordinates associated with $C$ and so on.

The following example illustrates what one might expect for a general variety.

\begin{example}\label{ex: scroll}
    Let $P$ be the lattice polygon depicted in Figure~\ref{fig: scroll}, and $X_P \subseteq \mathbb{P}^8$ denote the corresponding projective toric variety.  
    \begin{figure}[H]
            \begin{tikzpicture}[style = thick, scale = .6]
\tikzstyle{dot}=[circle,fill=black,inner sep=2 pt];
\begin{scope}[shift = {(0,0)}]
\node  [dot,label = below:{(0,0)}] (0) at (0,0) {};
\node  [dot] (1) at (1,0) {};
\node  [dot] (2) at (2,0) {};
\node  [dot] (3) at (3,0) {};
\node  [dot, label = below:{(4,0)}] (4) at (4,0) {};
\node  [dot] (5) at (0,1) {};
\node  [dot] (6) at (1,1) {};
\node  [dot] (7) at (2,1) {};
\node  [dot] (8) at (3,1) {};

\draw (0) -- (1) --(2) --(3)--(4) -- (8) -- (7)--(6)--(5)--(0);
\node at (-1.5,1) {$P = $};
\end{scope}
\begin{scope}[shift = {(10,0)}]
\node  [dot,label = below:{$z_0$}] (0) at (0,0) {};
\node  [dot,label = below:{$z_1$}] (1) at (1,0) {};
\node  [dot,label = below:{$z_2$}] (2) at (2,0) {};
\node  [dot,label = below:{$z_3$}] (3) at (3,0) {};
\node  [dot, label = below:{$z_4$}] (4) at (4,0) {};
\node  [dot,label = above:{$z_5$}] (5) at (0,1) {};
\node  [dot,label = above:{$z_6$}] (6) at (1,1) {};
\node  [dot,label = above:{$z_7$}] (7) at (2,1) {};
\node  [dot,label = above:{$z_8$}] (8) at (3,1) {};

\draw (0) -- (1) --(2) --(3)--(4) -- (8) -- (7)--(6)--(5)--(0);
\end{scope}
        \end{tikzpicture}
        \caption{The lattice polygon $P$ and the corresponding coordinates.}
        \label{fig: scroll}
    \end{figure}
    
    This is the rational normal scroll joining rational normal curves of degrees 3 and 4.  The ideal $\II(X_P)$ is generated by the $2 \times 2$ minors of the matrix 
    \[A = \begin{bmatrix}
    z_0 & z_1 & z_2 & z_5 & z_6\\
    z_1 & z_2 & z_3 & z_6 & z_7\\
    z_2 & z_3 & z_4 & z_7 & z_8
\end{bmatrix},\] and the ideal of $X_P^{\{2\}}$, the secant variety of lines spanned by pairs of points of $X_P$, is generated by the $3 \times 3$ minors of $A$. 

Examining the bases of the algebraic matroids of the affine cones $\tilde{X}_P$ and $\tilde{X}_{P}^{\{2\}}$ will help us gain intuition for what is to come.  For example, both $\{z_0, z_1, z_5\}$ and $\{z_2, z_6, z_7\}$ are bases of $M(\tilde{X}_P),$ and their union is a basis of $M( \tilde{X}_{P}^{\{2\}})$. In fact, {\em every} basis of $M(\tilde{X}_{P}^{\{2\}})$ is the union of disjoint bases of $M(\tilde{X}_P).$ We also note that neither the algebraic matroid of $X_P$ nor the algebraic matroid of its secant variety are uniform matroids.  Indeed, not every 3-element subset of coordinates is a basis of $M(X_P)$ as $z_1z_3-z_2^2$ (the $2 \times 2$ minor in the lower left corner) is in $\II(X_P)$, so $\{z_1, z_2, z_3\}$ is dependent.  Moreover, the determinant of the first three columns is in $\II(M(\tilde{X}_P)),$ so any 6-element set containing the 5 variables in the support of this determinant is also dependent.
\exqed
\end{example}

Example~\ref{ex: scroll} suggests that the algebraic matroid of the $s$th secant variety of a variety $X$ 
should have bases that are the disjoint union of $s$ bases of $M(X).$ While this is often true in 
sufficiently general coordinates, it cannot be true if $\dim \tilde{X}^{\{s\}} < s\dim \tilde{X}.$  In fact, 
Theorem~\ref{introtheorem: union} shows that this is the general obstruction.  
Theorem~\ref{introtheorem: subunion} shows there is a general relationship between the algebraic 
matroid of a join and the algebraic matroids of its summands.

To state our main results, we require some additional language from
combinatorics. Recall that if $M_1,\ldots,M_s$ are matroids on ground
sets $E_i$,
their matroid union $M_1 \vee \cdots \vee M_s$ is the matroid on
$E = E_1 \cup \cdots \cup E_s$ whose independent sets are
\[
  \{ I \subseteq E : I = I_1 \cup \cdots \cup I_s,\ \text{each $I_j$
  independent in $M_j$} \}.
\]

\medskip

Our first theorem establishes that the matroid union provides a natural
upper bound, in the weak order on matroids, for the algebraic matroid of a
join of irreducible affine cones.

\begin{introtheorem}[Sub-union Theorem]\label{introtheorem: subunion}
  Let $\kk$ be an algebraically closed field of characteristic zero. 
  If $X_1, \ldots, X_s\subseteq \mathbb{A}_\kk^n$ are irreducible
  affine cones with join $X$,  then
  \[
    M(X) \preceq M(X_1) \vee \cdots \vee M(X_s).
  \]
\end{introtheorem}
The hypothesis on $\kk$ is for technical reasons that will become clear in the
proof.  When equality holds in Theorem~\ref{introtheorem: subunion}
we say that $M(X)$ is a
{\em Terracini union}.  We will see that the Terracini union property
is determined by certain
projections.
For $E\subseteq [n]$, we denote by $\pi_E : \mathbb{A}_\kk^n\to
\mathbb{A}_\kk^E$
the linear projection to the coordinate subspace indexed by $E$.
The Sub-union theorem tells us that the algebraic matroid of a join is always  contained in the union of
the matroids of its factors.
The central problem is to determine when equality holds, i.e.~when $M(X)$ is a
\emph{Terracini union}.

\medskip
Our second main theorem characterizes when equality holds in the
Sub-union Theorem.

\begin{introtheorem}[Union Theorem]\label{introtheorem: union}
  Let $\kk$ be an algebraically closed field of characteristic zero. 
  If $X_1, \ldots, X_s\subseteq \mathbb{A}_\kk^n$ are irreducible affine cones
  with join $X$,
  then $M(X)$ is a Terracini union
  if and only if there does not exist a basis $B$ of $M(X_1) \vee
  \cdots \vee M(X_s)$
  such that the join
  \[
    \overline{\pi_B(X_1)} + \cdots + \overline{\pi_B(X_s)}
  \]
  is defective.
\end{introtheorem}

Our work was motivated by questions originating in rigidity theory.
Example~\ref{exm: rigidity} below shows that the matroids that arise in
rigidity theory do not satisfy the Terracini union property.  Indeed,
the failure of the Terracini union property is a sign that of
interesting combinatorics! 

\begin{example}[Generic rigidity]\label{exm: rigidity}
  The Cayley--Menger variety
  ${\rm CM}_{d,n}$ studied in rigidity theory
  is
  the Zariski closure of the image of the map
  $\varphi:(\mathbb{C}^d)^n \to \mathbb{C}^{\binom{n}2}$ given by
  \[\varphi(p_1, \ldots, p_n) = (p_{i1}-p_{j1})^2+ \cdots +
  (p_{id}-p_{jd})^2\] whose restriction to $(\mathbb{R}^d)^n$ gives the
  squares of pairwise distances among $n$ points in $\mathbb{R}^d$
  \cite{B02}.
  From the form of $\varphi$, we see that
  \[
    {\rm CM}_{d,n} = {\rm CM}_{1,n} + \cdots + {\rm CM}_{1,n}
  \]
  is the $d$-fold join of  ${\rm CM}_{1,n}$ with itself.  In other words,
  ${\rm CM}_{d,n}$ is the $d$-th secant variety $({\rm
  CM}_{1,n})^{\{d\}}$ \cite{GHT}.  However, we note that $\dim {\rm CM}_{1,n} =
  n-1$ whereas $\dim {\rm CM}_{d,n} = dn-\binom{d+1}2< d(n-1)$ if $d>1.$

  In terms of combinatorics, $M({\rm CM}_{1,n})$ is well-known to be
  isomorphic to the graphic matroid of
  $K_n$. More generally, a basis of the rigidity matroid $M({\rm CM}_{d,n})$
  corresponds to a generically minimally rigid graph in dimension $d$
  with $n$ vertices.  Classifying these combinatorially is a
  notable open problem for $d\ge 3$ (see \cite{KJT22a,KJT22b} for
  recent progress on $d=3$). An
  exercise with the differential of $\varphi$ (e.g., \cite{WW83})
  also shows that
  any independent set in $M({\rm CM}_{d,n})$ can be partitioned
  into $d$ forests.  However, counting the number of 
  edges in a basis shows that these $d$ forests cannot all be 
  spanning trees as one would expect if
  $M({\rm CM}_{d,n}) = M({\rm CM}_{1,n}^{\{d\}})$ were a Terracini union. In fact, the bases of $M({\rm CM}_{d,n})$ 
  are precisely the connected graphs $G = (V,E)$ with $|E| = dn-\binom{d+1}2$ so that 
  $\overline{\pi_E({\rm CM}_{1,n})}$ is not $d$-defective.
\exqed
\end{example}

Our results are related to work of \cite{draisma2008tropical} and
\cite{lmr} who studied
secant defectiveness.
Laface, Massarenti, and Rischter~\cite{lmr} analyze
non-defectiveness of secant varieties of toric varieties via Terracini's
lemma and tangent space computations. These computations are formulated
as a linear optimization problem which is the essence of the tropical
perspective
introduced by Draisma~\cite{draisma2008tropical} to study secant
defectiveness more generally. While compatible with their approach,
our framework
uncovers additional combinatorial obstructions: the matroid union
property can fail even in cases where the tangent-space rank test of
\cite{lmr} succeeds. Example~\ref{ex:Laface} in \S\ref{subsec:Laface}
illustrates this distinction.

\subsection*{Outline.}
Section~\ref{sec:prelim} sets notation and reviews the basic notions of
joins, secant varieties, matroids, and algebraic matroids.
We prove our main results,  Theorems~\ref{introtheorem: subunion} 
and~\ref{introtheorem: union}, in Section~\ref{sec:secant}.
Section~\ref{sec:exam} develops examples from statistics, rigidity theory,
and toric geometry, illustrating that equality is subtle: projections
and parameter
choices can break the Terracini property. In particular, Example~\ref{ex:Laface}
compares our framework with tangent-space methods of~\cite{lmr},
highlighting the additional combinatorial obstructions that our
approach detects.

\section{Preliminaries}\label{sec:prelim}
In this section we briefly set notation and review relevant
definitions and intuition having to do with
joins, secant varieties, matroids, and algebraic matroids.  For the rest of this paper, 
let $\kk$ be an algebraically closed field of characteristic zero.

\subsection{Preliminaries from algebraic geometry}  
Let $S = \kk[z_1, \ldots, z_n]$ and give $\kk^n$ 
the basis corresponding to the coordinates $z_i$.  
Let $X\subseteq \kk^n$ be an irreducible variety with the 
ideal $\II = \II(X) = \langle f_1, \ldots, f_m\rangle\subseteq S$.  
We review a package of results 
analogous to differential geometric ones downstream from Sard's theorem using 
\cite{vakil} as our main reference.  

\subsubsection{Cotangent spaces with a canonical frame}
The first thing we need is to set up some machinery 
involving cotangent spaces of embedded varieties with a 
fixed coordinate system.  This will provide a bridge 
between intrinsic algebro-geometric results, and matroid theory, 
which depends on the coordinate system.

Let $R = S/\II$.  Denote by $\Omega_{S/\kk}|_X$ the $R$-module 
$\Omega_{S/\kk}\otimes_S R$, and by ${\rm d}$ the universal 
derivation.  Concretely, $\Omega_{S/\kk}\otimes_S R$ is 
the free $R$-module generated by $\{{\rm d}z_i\}$ (see 
\cite[\S 21.2.2 and 21.2.3]{vakil}).  By \cite[Thm 21.2.12]{vakil},
at a smooth point $x\in X$ with maximal ideal $m$, we get the exact sequence of 
$\kk$-vector spaces
\[
    0\to \II/\II^2\otimes_{\mathcal{O}_{X,x}} \kk \xrightarrow{{\rm d}}
    \Omega_{S/\kk}|_X\otimes_{\mathcal{O}_{X,x}}\kk\to 
    \Omega_{R/\kk}\otimes_{\mathcal{O}_{X,x}} \kk\to 0,
\]
because the residue field is $\kk$.
Left-exactness comes from the smoothness of $x$ (see \cite[Thm 21.2.31 and Cor 21.2.33]{vakil}).  
Using \cite[Exercise 21.2.F]{vakil}, the previous exact sequence is 
\[
    0\to N^*_x X\to (\kk^n)^* \to T^*_x X\to 0,
\]
where we have identified $T^*_x \kk^n\cong (\kk^n)^*$ using the 
fixed global coordinate frame 
$\{{\rm d}z_i\}$, which is independent of $x$ due to triviality 
of the cotangent bundle of $\kk^n$.  This allows us to glue together conormal 
sequences for different $X\subseteq \kk^n$ along the middle term.  

We immediately get an embedding
$N^*_x X\hookrightarrow (\kk^n)^*$ that is 
canonical with respect to the fixed coordinates $z_i$.  If we 
denote by $\kk^E$ the coordinate subspace spanned by 
$\{z_i : i\in E\}$, we see that, for any $x\in \kk^E$, 
$T^*_x \kk^E\cong (\kk^E)^*$ with the term on the 
right equal to the linear span of $\{{\rm d}z_i : i\in E\}$.

\subsubsection{Sard's theorem}
The second ingredient we need relates to smoothness.
A smooth point $y\in Y$ is a regular value of a map $f : X\to Y$ if, 
for any $x\in f^{-1}(y)$, the map $df_x: T_xM \to T_yM$ is surjective.  We will need the corresponding characterization for cotangent spaces, that the pullback of the differential $({\rm d}f)^*_y : T^*_y Y\to T^*_x X$ 
is injective. The following is an algebraic variant of Sard's theorem.
\begin{lemma}[{\cite[Thm 21.6.6]{vakil}}]\label{lem: sard}
Let $\kk$ be algebraically closed field of characteristic zero and $f : X\to Y$, a 
dominant map of $\kk$-varieties.  Then the set of regular values of $f$ 
contains an open dense subset of $Y$.
\end{lemma}
We remark that, as discussed in \cite[Section 21.6]{vakil}, this kind of 
``generic smoothness in the target'' requires characteristic zero, 
whereas a weaker ``generic smoothness on the source'' needs only 
an algebraically closed field.  
Combining the ingredients in this section, 
we get a quick proof of the conormal version of Terracini's Lemma in 
Appendix~\ref{sec: terracini}.

\subsection{Secants and joins}
Although our motivation is from questions about projective varieties,
statements will be cleaner if we instead work with affine cones, which
are affine varieties defined by homogeneous prime ideals.  In this
section we provide notation and definitions in this context.

We define $X \subseteq \mathbb{A}^n$ to be an {\em affine cone} if it
is an algebraic variety defined by a homogeneous ideal
$\II(X) \subseteq S = \kk[z_1, \ldots, z_n]$.  Note that if
$X$ is an affine cone, then it has the property that whenever 
$x \in X$ and $r \in \kk$, it follows that $rx \in X.$
If $X_1, \ldots, X_s$ are  irreducible affine cones, we define their
{\em embedded join}
\[
X_1+\cdots+X_s = \overline{
  \{x_1+ \cdots + x_s \mid x_i\in X_i\}
}
\]
to be  the Zariski closure of the union of the affine subspaces
spanned by $s$ points, one from each cone $X_i.$  The
{\em expected dimension} of $X_1+\cdots +X_s$ is
$\min\{\sum_{i =1}^s \dim X_i,N\}$.
If the dimension of $X_1+\cdots +X_s$ is less than the expected dimension
we say that the join is \emph{defective}.

In the special case where $X = X_i$ for all $i$, we write $X^{\{s\}}
= X_1 + \cdots +X_s$ and call this the $s$th secant variety of $X.$  We
say that $X$ is $s$-defective if its dimension is less than
$\min\{ns, N\}$, where $n = \dim X.$

Since we will use linear algebraic methods, we record the following version of Terracini's Lemma for cotangent 
spaces.
\begin{lemma}\label{lem: product cotangent space}
Let $X_1, \ldots, X_s$ be irreducible varieties in $\mathbb{A}_\kk^n$.  Then, over an open dense set of 
$x = (x_1, \ldots, x_s)$ in $X = X_1 \times \cdots \times X_s$, 
\[
    T^*_x X \cong T^*_{x_1} X_1 \oplus \cdots \oplus T^*_{x_s} X_s,
\]
where the isomorphism is canonical from pulling back projections onto 
the factors.
\end{lemma}
\begin{proof}
Let $p_i:X_1\times \cdots \times X_s\to X_i$ be the projection 
to the $i$th factor. Because 
$p_i$ is dominant for each $i$, the set of regular values 
contains an open dense subset $U_i$.  The set $p_1^{-1}(U_1)\cap \cdots \cap 
p_s^{-1}(U_s)$ is open and dense in $X$, since it is a finite intersection 
of open dense sets.  Hence, there is an open dense set $U\subseteq X$ of 
points $x = (x_1, \ldots,x_s)$ such that each $x_i$ is a regular value of 
$p_i$.

Now fix a point $y = (y_1, \ldots, y_s)\in X$ and define 
$\iota_i : X_i \to X$ by $\iota_i(x_i) = (y_1, \ldots, x_i, \ldots, y_s)$
(replace $y_i$ by $x_i$ in $y$).  Hence $\iota_i$ is a right 
inverse of $p_i$ for each $i$ and, for $j\neq i$, 
$p_j\circ \iota_i$ is a constant map that sends every point 
to $y_j$.  This implies that 
$({\rm d}\iota_i)^*_{\iota_i(x_i)}({\rm d}p_i)^*_{x_i}$ 
is the identity on $T^*_{x_i} X_i$ for all $x_i\in X$ 
and that, for $j\neq i$, $({\rm d}\iota_i)^*_{\iota(x_i)}({\rm d}p_j)^*_{x_j}$ 
is zero, for all $x_i\in X_i$ and $x_j\in X_j$.

With this setup, let $x\in U$ be given.
Suppose that $\varphi_i, \varrho_i\in T^*_{x_i} X_i$ are such that 
\[
    ({\rm d}p_1)^*_{x_1}(\varphi_1) + \cdots + ({\rm d}p_s)^*_{x_s}(\varphi_s) = 
    ({\rm d}p_1)^*_{x_1}(\varrho_1) + \cdots + ({\rm d}p_s)^*_{x_s}(\varrho_s).
\]
Applying $({\rm d}\iota_i)^*_{\iota_i(x_i)}$ to both sides  implies that $\varphi_i = \varrho_i$.  
This shows that the image of the 
linear map $T^*_{x_1}X_1 \times \cdots \times T^*_{x_s}X_s \to T^*_x X$
with components $({\rm d}p_i)^*_{x_i}$ is isomorphic to
\[
    \operatorname{im} ({\rm d}p_1)^*_{x_1} \oplus \cdots \oplus \operatorname{im} ({\rm d}p_s)^*_{x_s}.
\]
Because each $x_i$ is a regular value of $p_i$, this image is canonically 
isomorphic (the sections $\iota_i$ depend on $y$, but the
$({\rm d}p_i)^*_{x_i}$  do not) to
\[
    T^*_{x_1} X_1 \oplus \cdots \oplus T^*_{x_s} X_s,
\]
and then, considering dimension, we get the result.
\end{proof}

\subsection{Matroids}To state our results, we need to recall some
general material about matroids.
There is a  {\em weak order} $M\preceq M'$ on matroids with a common
ground set $E$, in which $M\preceq M'$ means that every dependent set
in $M'$ is also dependent in $M$.  We say that $M \prec M'$ if
$M\preceq M'$ and there is a set that
is dependent in $M$ but not in $M'$.  If $M_1, \ldots, M_s$ are
matroids on ground sets $E_1, \ldots, E_s$, the
{\em matroid union} $M_1\vee \cdots \vee M_s$ is the matroid $M$ on
ground set $E_1 \cup \cdots \cup E_s$ that has as its independent sets:
\[
\{ I : I =  I_1 \cup \cdots \cup I_s\},
\]
where each $I_i$ is independent in $M_i$. When the $M_i$
are all equal to a matroid $M$, we write the $s$-fold union as $sM$.

\subsection{Algebraic matroids}
The algebraic matroid of $X$, denoted $M(X),$ is the data of which
subsets of variables in
$S$ are related in the ideal $\II(X)$.
\begin{definition}\label{def: algebraic matroid}
Let $Z = \{z_1, \ldots, z_n\}$ and $X\subseteq \mathbb{A}^n$
be an irreducible variety so that $\II(X)$ is a prime ideal
in $S = \kk[Z]$.  Define the
{\em algebraic matroid} of $X$, denoted $M(X)$, to be the
matroid with ground set $Z$ where $E \subseteq
Z$ is independent if
$\II(X) \cap \kk[E] = \langle 0\rangle.$
\end{definition}
We will frequently use the following geometric interpretation.
\begin{lemma}\label{lem: independent iff dominant}
In the setup of Definition~\ref{def: algebraic matroid}, $E\subseteq Z$ is independent in 
$M(X)$ if and only if the linear projection $\pi_B : X\to \kk^E$ is dominant.
\end{lemma}
\begin{proof}
The Closure Theorem implies that the vanishing ideal of $\pi_E(X)$ is $\II(X) \cap \kk[E]$, so 
the projection is dominant if and only if $\II(X) \cap \kk[E] = \langle 0\rangle.$
\end{proof}
The rank of $M(X)$ is equal to the dimension of $X$.
This provides a natural notion of the ``expected rank'' for algebraic
matroids of embedded joins, namely that the expected rank of $M(X_1+\cdots +X_s)$ is the expected dimension of $X_1+\cdots+X_s.$
To gain further intuition about what algebraic matroids capture,
we examine how coordinate changes affect the matroid
$M(X^{\{s\}})$.
\begin{example}\label{ex: P3 veronese}
To illustrate how coordinate changes affect algebraic matroids, we examine the image of the quadratic
Veronese map $\nu_2:\mathbb{P}^3 \to \mathbb{P}^9$ in three different
coordinate systems.  If we let the coordinates of $\nu_2$ be given by
monomials, then its image is defined by an ideal in
$\mathbb{C}[z_{ij} : 1\leq i <j\leq 5]$ generated by the $2 \times 2$
minors of the generic symmetric matrix:
\[
  A_1 =
  \begin{pmatrix}
    z_{15} & z_{12} & z_{13} & z_{14}\\
    z_{12} & z_{25} & z_{23} & z_{24}\\
    z_{13} & z_{23} & z_{35} & z_{34}\\
    z_{14} & z_{24} & z_{34} & z_{45}
\end{pmatrix}.\]
A linear change of coordinates results in the matrix
\[A_2 =
  \begin{pmatrix}
    2z_{15}& z_{15}+z_{25}-z_{12}& z_{15}+z_{35}-z_{13}&
    z_{15}+z_{45}-z_{14}\\
    z_{15}+z_{25}-z_{12} & 2z_{25} & z_{25}+z_{35}-z_{23}&
    z_{25}+z_{45}-z_{24}\\
    z_{15}+z_{35}-z_{13}& z_{25}+z_{35}-z_{23}& 2z_{35} &
    z_{35}+z_{45}-z_{34}\\
    z_{15}+z_{45}-z_{14}& z_{25}+z_{45}-z_{24}& z_{35}+z_{45}-z_{34} & 2z_{45}
  \end{pmatrix},
\]
whose $2\times 2$ minors define an isomorphic variety.  Finally, let
$A_3 = f(A_1)$, where $f$ is a general linear change of coordinates.

With this notation we can describe the ideals of the quadratic Veronese surface and its
secant varieties in three different coordinate systems.
Define $\II_k(A_i)$ to be the ideal generated by the $k \times k$
minors of $A_i$,
and set $X_i = V(\II_2(A_i))$, $X_i^{\{2\}} = V(\II_3(A_i))$, and
$X_i^{\{3\}} = V(\II_4(A_i)).$
Each $X_i$ is defined by the $2\times 2$ minors of $A_i$,
$X_i^{\{2\}}$ is its
variety of 2-secant lines, and $X_i^{\{3\}}$ is its variety of
3-secant planes. Note that $X_1$ is a
toric variety, $X_2$ is the Cayley-Menger variety ${\rm CM}_{1,5}$, and
$X_3$ is isomorphic to
both of these with generic coordinates.

The algebraic matroid of a variety captures dependencies among
coordinates, and since the polynomial relations on these varieties
are different, we expect them to have different algebraic matroids.
Thus, although
the three varieties are isomorphic, their corresponding algebraic
matroids are not (until we reach $X_i^{\{3\}}$) as shown in the
following table.
\begin{table}[h!]
  \centering
  \begin{tabular}{|c|c|c|c|}
    \hline
    $X$ &  \# bases $M(X)$ & \# bases $M(X^{\{2\}})$ & \# bases
    $M(X^{\{3\}})$\\
    \hline
    $X_1$ & 141 & 104 & 10 \\
    \hline
    $X_2$ & 125 & 100 & 10\\
    \hline
    $X_3$ & 210 & 120 & 10\\
    \hline
  \end{tabular}
  \caption{Counts of the bases of the algebraic matroid of the
    quadratic embedding of
    $\mathbb{P}^3$ and its secant varieties with three different
  coordinate systems.}
  \label{tab:my_label}
\end{table}
Note that $210 = \binom{10}{4}$ and $120 = \binom{10}{7},$ so the
last row of the table
indicates that the algebraic matroid of $X_3$ and its secant
varieties are uniform matroids.
We want to emphasize that $M(X_1)\neq M(X_3)$ and $M(X_2) \neq
M(X_3)$ showing that $X_1$ and
$X_2$ have special coordinates resulting in many fewer independent
sets than in
the generic case.  Additionally, in each case $\rank M(X_i) = 4,$ and
the rank of the matroid union $M(X_i) \vee M(X_i)$ is 8,
while the rank of $M(X_i^{\{2\}})$ is 7.
\exqed
\end{example}

The proof of Theorem~\ref{introtheorem: subunion} relies on linearizing
$M(X)$.  The following is a folklore corollary of \cite{ingleton}, 
but we give a quick proof for completeness.
\begin{theorem}\label{thm: ingleton}
Let $X \subseteq \mathbb{A}_{\kk}^n$ be an irreducible affine variety.
If $\kk$ is closed
and has characteristic zero, then there is an open, dense subset
$U\subseteq X$
such that $M(X)$ is isomorphic to the $\kk$-linear matroid
of the images ${\rm d} z_i$ of the coordinate functions $z_i$ in the Zariski
cotangent space at any point $x\in U$.
\end{theorem}
\begin{proof}
Let $Z = \{z_1, \ldots, z_n\}$.  Let $B$ be a basis of $M(X)$.  By Lemma~\ref{lem: independent iff dominant}, 
$\pi_B$ is dominant, so there is an open dense subset $V_B\subseteq \kk^B$ of regular values 
by Lemma~\ref{lem: sard}. 
Set $U_B = \pi^{-1}_B(V_B)$; this is an open dense subset of $X$ by continuity and 
irreducibility of $X$.  For any $x\in U_B$, the differential pullback 
$({\rm d}\pi_B)^*_x : T^*_{\pi_B(x)} \kk^B \to T^*_x X$ 
is injective, and considering dimension, a linear isomorphism.  
By the identification of $T^*_{\pi_B(x)}$ with $(\kk^B)^*$, 
the images of $\{{\rm d} z_i : i\in B\}$ in 
$T^*_x X$ are linearly independent, and considering dimension, a basis. 
The set $U$ in the conclusion is obtained by taking the intersection of $U_B$ as $B$ 
varies over all bases of $M(X)$.
\end{proof}
See \cite{ rosen2025linearizing} for a more elementary proof and a
discussion of the
relationship with \cite{ingleton}.  If a variety is the image of
another variety, then the linear matroid in Theorem~\ref{thm:
ingleton} can also be obtained from the image of the differential at
a suitably general point.  We illustrate these ideas in the following example.

\begin{example}
Let $\varphi: \mathbb{C}^4 \to \mathbb{C}^4$ be given by
$\varphi(s,t,u,v) = (su,sv, tu, tv).$  The Zariski closure of the
image of $\varphi$ is the variety $X$ with $\II(X) =
\langle z_1z_4-z_2z_3 \rangle.$  The matroid $M(X)$ is the uniform
matroid of rank 2 on 4 elements.  From Theorem~\ref{thm: ingleton},
at a generic point of $X$, the algebraic matroid $M(X)$ is isomorphic
to the linear
matroid on the ${\rm d}z_i$ with unique circuit
\[
  z_4{\rm d}z_1+z_1{\rm d}z_4-z_2{\rm d}z_3-z_3{\rm d}z_2.
\]
Here it is enough to choose a point on $X$ where all coordinates are nonzero.

We can compute the same matroid from the differential of $\varphi$
again assuming the input is suitably generic.  For example, we have
\[ {\rm d} \varphi =
  \begin{bmatrix}
    u & 0 & s & 0\\
    v & 0 & 0 & s\\
    0 & u & t & 0\\
    0 & v & 0 & t
\end{bmatrix}.\]  At a suitably general point, say $p = (1,1,1,1),$
${\rm d}\varphi_p$ is a rank 2 matrix in which every pair of rows is
linearly independent, giving the uniform matroid of rank 2 on 4 elements..
However, at $q = (1,0,1,0),$ the last row of ${\rm d}\varphi_q$ is zero,
so the linear matroid on the rows of $d\varphi_q$ is not uniform.
\exqed
\end{example}

\subsection{A linear algebra lemma}
We need a standard result from combinatorial linear algebra.  It 
generalizes the Laplace expansions from \cite{WW83}, which have been 
used many times in rigidity theory, including in \cite{cruickshank2023identifiability}. %
 We begin by introducing some notation that we will use throughout the paper.  
Let $V$ be a vector space with basis $\{v_1, \ldots, v_n\}.$ If $E \subseteq[n],$ define 
$V_E = \Span \{v_i \mid i \in E\},$ and let $\pi_{E}:V \to V_E$ be the 
linear projection to $V_E$ with kernel $V_{\bar{E}}$.
\begin{lemma}\label{lem: partition}
Let $T :V \to  U_1 \oplus \cdots \oplus U_s$
be a linear isomorphism between $n$-dimensional
vector spaces and let $p_i:U_1\oplus \cdots \oplus U_s \to U_i$ be the $i$th projection. 
If $\{v_i\}$ is any basis of $V$,
there is a partition $\{E_1, \ldots, E_s\}$ of $[n]$ so that, for
each $i$, the linear maps
$p_i \circ T|_{V_{E_i}} : V_{E_i} \to U_i$ and $ \pi_{E_i} \circ T^{-1}|_{U_i} :U_i \to V_{E_i}$ are  isomorphisms. 
\end{lemma}
\begin{proof}
The general case follows by induction from the statement when $s=2$,
so to simplify notation, we look at $T:V \to U_1\oplus U_2$.  
Let $E_1 \subseteq [n]$ be a minimal set so that $p_1 \circ T|_{V_{E_1}}$ 
is surjective and let $E_2 = \overline{E}$.  By minimality, 
$p_1 \circ T|_{V_{E_1}}$ must be an isomorphism.  We claim that 
$p_2 \circ T|_{V_{E_2}}$ is injective.  Indeed, suppose that 
$v\in V_{E_2}$ such that ${\bf 0} = p_2 \circ T|_{V_{E_2}}(v) = p_2(T(v)).$ 
Then $T(v) \in \ker p_2 = U_1,$ and since $U_1 \cap U_2 = {\bf 0},$ it must 
be that $v = {\bf 0}.$  Since $|E_2| = \dim U_2$, if  
$p_2 \circ T_{V_{E_2}}$ is injective, it is an isomorphism.  
Since $\pi_{E_i} \circ T^{-1}|_{U_i} :U_i \to V_{E_i}$ is the 
inverse of $p_i \circ T|_{V_{E_i}} : V_{E_i} \to U_i$, the result follows.
\end{proof}

\section{Algebraic matroids and embedded joins}\label{sec:secant}
In this section, we prove our main theorems.  
\begin{theorem}[Sub-union Theorem]\label{thm: subunion}
Let $\kk$ be a field with $\kk = \overline{\kk}$ and
$\operatorname{char}\kk = 0$.
If $X_1, \ldots, X_s\subseteq \mathbb{A}_\kk^N$ are irreducible
affine cones,  then
\[
  M(X_1 + \cdots + X_s) \preceq M(X_1) \vee \cdots \vee M(X_s).
\]
\end{theorem}
\begin{proof}
Let $X = X_1+ \cdots + X_s$ and denote by $\sigma : X_1 \times \cdots \times X_s\to X$
the map that adds points.  For each $X_i$, by Theorem~\ref{thm: ingleton}, 
there is an open, dense subset $U_i\subseteq X_i$ such that, if $x_i\in U_i$, 
then $M(X_i)$ is isomorphic to the linear matroid of the coordinate differentials 
${\rm d}z_j$ in $T^*_{x_i} X_i$.  By irreducibility of the $X_i$ and $X$, 
$\sigma(U_1\times \cdots \times U_s)$, as a constructible subset of $X$ of the same dimension, 
contains an open dense subset $U'$ of $X$.
Another application of Theorem~\ref{thm: ingleton} gives an open dense subset 
$U''$ of $X$ such that, if $x\in U''$, then $M(X)$ is isomorphic to the linear 
matroid of the coordinate differentials ${\rm d}z_j$ in $T^*_{x} X$.  Finally, 
Lemma~\ref{lem: sard} implies that $X$ contains an open dense subset $U'''$ of 
regular values of $\sigma$.  Let $U = U'\cap U''\cap U'''$.  This is 
also open and dense in $X$.

Let $U\ni x = x_1 + \cdots + x_s$ be given.  Fix a basis $E$ of $M(X)$.  Then, because 
$x\in U$, $y = \pi_E\circ\sigma(x_1, \ldots, x_s)$ is a regular value of the composed 
map.  Using Lemma~\ref{lem: product cotangent space}, the map $({\rm d}\sigma)^*_x({\rm d}\pi_E)^*_y$
gives a canonical linear embedding 
\[
   ({\rm d}\sigma)^*_x({\rm d}\pi_E)^*_y: (\kk^E)^* \to T^*_{x_1} X_1 \oplus \cdots \oplus T^*_{x_s} X_s.
\]
We now select subspaces $W_i\subseteq T^*_{x_i} X_i$ so that 
\[
    ({\rm d}\sigma)^*_x({\rm d}\pi_E)^*_y: (\kk^E)^* \to W_1 \oplus \cdots \oplus W_s
\]
is an isomorphism, as follows\footnote{For readers familiar with rigidity theory, we note
that this step corresponds to ``tying down'' the rigidity matrix in \cite{WW83}.}.  Let $W$ be the image of $({\rm d}\sigma)^*_x({\rm d}\pi_E)^*_y$.
Greedily select a basis for $W$ by adding vectors in $T^*_{x_1} X_1$ until it is not possible 
to extend the current linearly independent set with another vector in $T^*_{x_1} X_1$.  Then 
move on to $T^*_{x_2} X_2$ repeating the same procedure.  This process will produce a basis for $W$, 
and we define $W_i$ as the linear span of the basis vectors in $T^*_{x_i} X_i$.  Since the $W_i$ are 
contained in independent subspaces, we have $W =  W_1 \oplus \cdots \oplus W_s$.

Now we apply Lemma~\ref{lem: partition} to $({\rm d}\sigma)^*_x({\rm d}\pi_E)^*_y$ 
to get a partition $E = E_1\cup \cdots \cup E_s$ of 
$E$ so that, for each $i$, 
$p_i\circ ({\rm d}\sigma)^*_x({\rm d}\pi_E)^*_y|_{(\kk^{E_i})^*}  :(\kk^{E_i})^* \to W_i$
is an isomorphism.

Let 
$\iota_i : (\kk^{E_i})^*\hookrightarrow (\kk^E)^*$ denote the 
inclusion.  We claim that 
\[
    ({\rm d}\sigma)^*_x({\rm d}\pi_E)^*_y\circ\iota_i = ({\rm d}\pi_{E_i})^*_{x_i}.
\]
By identification of $(\kk^{E_i})^*$ with its image in $(\kk^E)^*$
by the global coordinate frame,
we have $\varphi\in T^*_{x_i} X_i$ iff $\varphi\in (\kk^{E_i})^*$.
Let $\varphi \in (\kk^{E_i})^*$ be given.
By definition, 
$({\rm d}\sigma)^*_x({\rm d}\pi_E)^*_y\circ \iota_i(\varphi) = ({\rm d}\sigma)^*_x({\rm d}\pi_E)^*_y(\varphi)$, 
which is in $W_1 \oplus \cdots\oplus W_s$.  Because the cotangent 
bundle of $\kk^n$ is trivial, ${\rm d}\sigma$ is also the addition map, 
and so the $i$th component of $({\rm d}\sigma)^*_x({\rm d}\pi_E)^*_y(\varphi)$ is simply 
$({\rm d}\pi_{E_i})^*_{x_i}(\varphi)$, which is what we wanted to prove.

Noting that, because the inclusion $\iota_i$ is canonical in the fixed coordinate frame on
$(\kk^n)^*$, 
\[
    ({\rm d}\sigma)^*_x({\rm d}\pi_E)^*_y\circ\iota_i = ({\rm d}\sigma)^*_x({\rm d}\pi_E)^*_y|_{(\kk^{E_i})^*},
\]
from which we get that 
\[
    ({\rm d}\pi_{E_i})^*_{x_i} : (\kk^{E_i})^* \to T^*_{x_i} X_i
\]
is injective.

By Theorem~\ref{thm: ingleton}, which applies because 
$x_i$ is a regular value of $\pi_{E_i}$, $E_i$ is independent in 
$M(X_i)$.  As $i$ was arbitrary, $E$ is independent in 
$M(X_1)\vee \cdots \vee M(X_s)$ as desired.
\end{proof}
If we wish to determine whether $M(X_1+\cdots+X_s)$ is a Terracini union, 
the situation is somewhat subtle.
The following example shows that the failure of $M(X^{\{s\}})$ to equal $sM(X)$ 
does not require $X$ itself to be $s$-defective.
\begin{example}\label{ex: veronese}
Let $X$ be the affine cone in $\mathbb{A}^{10}$ arising from the
image of the cubic Veronese map  $\nu_3:\mathbb{P}^2
\to \mathbb{P}^9$ defined by
\[
  \nu_3([1:s:t]) = [1:s:s^2:s^3:t:st:s^2t:t^2:st^2:t^3].
\]
The matroid $M(X)$ is a rank 3 matroid with 105 bases, and
$M(X^{\{2\}})$ is a rank 6 matroid with 207 bases.
The matroid $2M(X)$ is the uniform matroid of rank 6 on 10 elements
and has 210 bases.
The three subsets of $E$ of cardinality 6 that
fail to be bases of the matroid $M(X^{\{2\}})$ are:
\[
  \{z_0, z_1, z_2, z_4, z_5, z_7\}, \{z_1, z_2, z_3, z_5, z_6,
  z_8\}, \{z_4, z_5, z_6, z_7, z_8, z_9\}.
\]
\vspace{-1cm}
\begin{figure}[H]
  \centering
  \includegraphics[scale = .8]{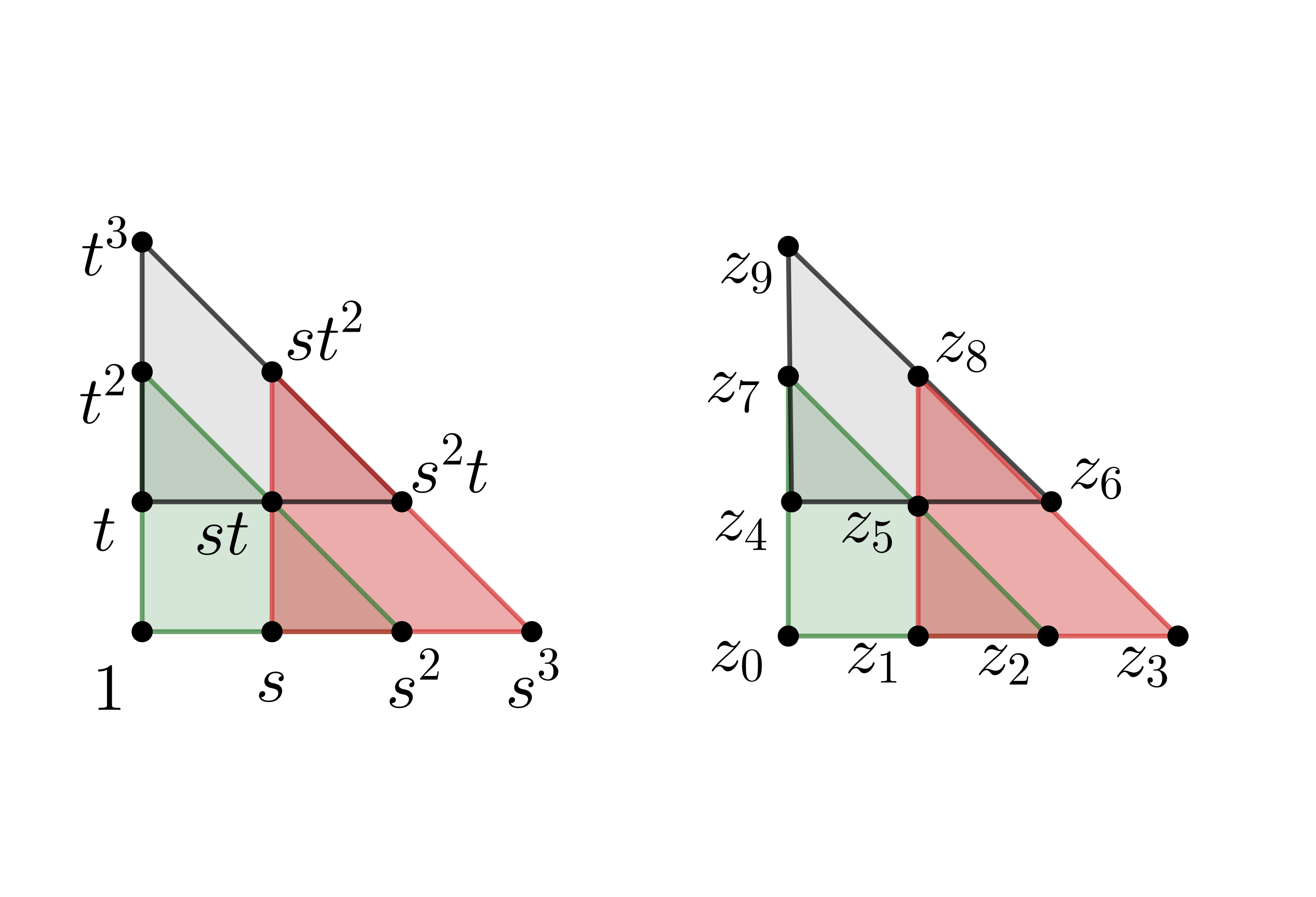}
  \caption{The Veronese embedding of $\mathbb{P}^2$ by cubics has
  three projections to the quadratic Veronese.}
  \label{fig: cubicVeronese}
\end{figure}
In each case, the monomials associated to the coordinates
parameterize $\nu_2(\mathbb{P}^2),$ the quadratic embedding of
$\mathbb{P}^2 \hookrightarrow \mathbb{P}^5$ as we can see from
Figure~\ref{fig: cubicVeronese}.   Moreover, we know that the
variety embedded by these monomials is defective.\exqed
\end{example}

The following theorem 
shows that in general the phenomenon observed in Example~\ref{ex: veronese} exactly 
characterizes when the algebraic matroid of a join fails to be a Terracini union.
\begin{theorem}[Union Theorem]\label{thm: union equality}
Let $\kk$ be a field with $\kk = \overline{\kk}$ and
$\operatorname{char}\kk = 0$.
If $X_1, \ldots, X_s\subseteq \mathbb{A}_\kk^n$ are irreducible affine cones
with join $X$,
then $M(X)$ is a Terracini union
if and only if, there does not exist a basis $B$ of $M(X_1) \vee
\cdots \vee M(X_s)$
such that the join
\[
  \overline{\pi_B(X_1)} + \cdots + \overline{\pi_B(X_s)}
\]
is defective.
\end{theorem}
\begin{proof}
Let $X = X_1 + \cdots + X_s$ and $M_\vee =  M(X_1) \vee \cdots \vee M(X_s)$.
Suppose that $M(X) \neq M_\vee$.  We first show that there is a
basis $B$ of $M_\vee$ such that
\[
  \overline{\pi_B(X_1)} + \cdots + \overline{\pi_B(X_s)}
\]
is defective.

By Theorem~\ref{thm: subunion} and the hypothesis that $M(X)\neq M_\vee$,
we get $M(X)\prec M_\vee$, so there is a basis $B$ of $M_\vee$
that is dependent in $M(X)$.  For this $B$, the map 
$\pi_B : X\to \kk^B$ is not dominant by Lemma~\ref{lem: independent iff dominant}. Using linearity of $\pi_B$, 
we conclude that
\[
    \dim \overline{\pi_B(X)} = \dim \overline{ \pi_B(X_1) + \cdots + \dim \pi_B(X_s)} < |B|.
\]
Because the join
\[
    Y = \overline{\pi_B(X_1)} + \cdots + \overline{\pi_B(X_s)}\subseteq \overline{\pi_B(X)}
\]
is dense in  $\overline{\pi_B(X)}$, the two sets have the same dimension.  
To finish this direction, we compute the expected dimension of $Y$. Because $B$ is a
base of $M_\vee$, there is a partition of $B$ into sets $B_i$
that are, for each $i$, independent in $M(X_i)$.  Since $B_i$ is 
independent in $M(X_i)$, for each $i$, $\pi_{B_i} : X_i\to \kk^{B_i}$ is
dominant, so
$\dim \overline{\pi_{B}(X_i)} \ge \dim \overline{\pi_{B_i}(X_i)} = |B_i|$.
Adding up these inequalities, the expected dimension of
$Y$ is at least $|B| = \dim \kk^B$.  Since $Y\subseteq \kk^B$,
equality holds.  As $\dim Y < |B|$, it is defective.

For the other direction, we suppose that there is a basis
$B$ of $M_\vee$ such that
\[
  Y = \overline{\pi_B(X_1)} + \cdots + \overline{\pi_B(X_s)}
\]
is defective.  As $Y$ is dense in $\pi_B(X)$, this implies that 
$\dim \pi_B(X)$ is less than the expected dimension, which is at most 
$B$.  Hence $\pi_B : X\to \kk^B$ is not dominant, so $B$ is 
dependent in $M(X)$ by Lemma~\ref{lem: independent iff dominant}, 
which gives $M(X) \neq M_\vee$.
\end{proof}
This next example shows that it is possible for a defective
join to have a matroid that is a Terracini union.
\begin{example}\label{ex: coloop extension}
Let $S = \kk[z_1,z_2, z_3, z_4, z_5]$ and $S' = \kk[z_1,z_2, z_3,
z_4, z_5, z_6]$.  Consider
\[
  A =
  \begin{bmatrix}
    z_1 & z_ 2& z_3 \\
    z_2 & z_3 & z_4 \\
    z_3 & z_4 & z_5
  \end{bmatrix}
\]
and let $g: \mathbb{A}^5 \to \mathbb{A}^5$ be a general linear change
of coordinates.  Let $X$ be the variety whose ideal
$\II(X)\subseteq S$ is generated by the $2\times 2$ minors
of $g(A)$ and $X' \subseteq \mathbb{A}^6$ be the variety whose
ideal $\II(X') \subseteq S'$ has the same generators.

In geometric terms, $\II(X)$ is the homogeneous ideal of a rational
normal curve of degree 4 in $\mathbb{P}^4.$  So, $\dim X^{\{2\}} = 2
\dim X^{\{2\}} = 4$, and $X$ is non-defective. The geometric
interpretation of $X'$ is that it is a cone over $X$, and $\dim X' =
\dim X+1 = 3.$  Moreover,  $(X')^{\{2\}}$ is also a cone
over $X^{\{2\}}$, and
$\dim (X')^{\{2\}} = 1 + \dim X^{\{2\}} = 5 <6,$ so $X'$ is defective

On the matroidal side, $M(X)$ is the rank $2$ uniform matroid on
$\{z_1, \ldots, z_5\}$ and  $M(X')$ is the extension of this matroid
by the coloop $z_6.$
Direct computations show that $M((X')^{\{2\}})$
is the extension of the rank $4$ uniform matroid on $\{z_1, \ldots,
z_5\}$ by the
coloop $z_6$.   From the definition of
$2M(X)$, we
get that $2M(X)$ is the extension of the rank $4$ uniform matroid
on $\{z_1, \ldots, z_5\}$ by the
coloop $z_6$, so $2M(X') = M((X')^{\{2\}})$.

Therefore, we have shown that it may be the case that $X'$ is
defective and yet $2M(X') = M((X')^{\{2\}})$.  We note that one can
also use Theorem~\ref{thm: union equality} to see that $M((X')^{\{2\}}) = 2M(X')$
and leave this as an exercise for the reader.
\exqed
\end{example}
\begin{remark}
Note that $\rank (M(X_1)\vee \cdots \vee M(X_n))$ may not be the same
as the expected
rank of $M(X_1+\cdots +X_n)$.  Indeed, in Example~\ref{ex: coloop extension}
$\rank 2M(X') = 5,$ but the expected rank of $M((X')^{\{2\}})$ is 6.
\exqed
\end{remark}

In general, we can also investigate the rank of an arbitrary subset
$E\subseteq Z$ in $M(X_1+\cdots +X_n)$.  
Using Theorem~\ref{thm: subunion} we see that the rank of 
$E$ is the maximum of $\rank_{M(X_1+\cdots +X_n)}(A)$, 
where $A$ ranges over subsets of $E$ that are independent in 
$M(X_1)\vee \cdots \vee M(X_n)$.
Furthermore, we know that if $A$ is independent in
$M(X_1+\cdots +X_n)$ then $|A| \leq \dim X_1+\cdots +\dim X_n,$
which provides another restriction on the subsets $A\subseteq E$ that we need to check in 
order to compute $\rank (E)$.

One might attempt to estimate $\rank_{M(X_1+\cdots +X_n)}(E)$ using the join defect of 
$X_1+\cdots +X_n,$. However, Example~\ref{ex: bolker-roth} shows that  $\pi_E(X_1+\cdots +X_n)$ 
can have a larger join defect, which in turn causes the rank to be smaller than expected 
based on the join defect of $X_1+\cdots +X_n$.

\begin{example}\label{ex: bolker-roth}
Let us recall some facts about secant varieties of Veronese and Segre varieties.  
Let $S(n;r)$ denote the determinantal variety of $n \times n$ symmetric matrices of rank
at most $r$ (an affine cone over a secant variety of a Veronese variety), 
and let $\Sigma(m,n;r)$ denote the variety of $m \times n$ matrices of rank at most 
$r \le \min\{m,n\}$ (an affine cone over a secant variety of a Segre variety).
When $\operatorname{char}\kk \neq 2$, we have 
$S(n;r) = S(n;1)^{\{r\}}$ by elementary linear algebra, 
and similarly $\Sigma(m,n;r) = \Sigma(m,n;1)^{\{r\}}$.
These $r$-secants have dimensions 
$\dim S(n;r) = rn - \binom{r}{2}$ and 
$\dim \Sigma(m,n;r) = r(m + n - r)$, respectively.

Let $A = (a_{ij})$ be the $8 \times 8$ generic symmetric matrix.  
Then $S(8;1)$ is defined by the $2 \times 2$ minors of $A$, and 
$S(8;1)^{\{2\}} = S(8;2)$ is defined by the $3 \times 3$ minors of $A$.  
Since $\dim S(8;1) = 8$ and $\dim S(8;2) = 2 \cdot 8 - \binom{2}{2} = 15 < 2 \dim S(8;1) = 16$, 
the $2$-secant defect of $S(8;1)$ is $1$.

Now let $A'$ be the $4 \times 4$ upper right-hand block of $A$, and let $E$ denote the set of entries~in~$A'$.  
Since $A'$ is a generic $4 \times 4$ matrix, we have $\overline{\pi_E(S(8;1))} = \Sigma(4,4;1)$.  
The Segre variety $\Sigma(4,4;1)$ has dimension $7$, and 
$\dim \Sigma(4,4;1)^{\{2\}} = \dim \Sigma(4,4;2) = 2(4+4-2) = 12$.  
However, the expected dimension of $\Sigma(4,4;1)^{\{2\}}$ is $2 \cdot 7 = 14$, 
hence the $2$-secant defect of $\overline{\pi_E(S(8;1))}$ is $2$.

Therefore, the rank of $E$ in $M(S(8;1)^{\{2\}})$ is $12$, which is one less than 
we would have predicted based on the secant defect of $S(8;1)$. 
\exqed
\end{example}

\section{Applications}\label{sec:exam}
In this section we present applications of the Union Theorem (Theorem~\ref{thm: union
equality}).  Many of the examples discussed here arise in projective geometry, where we work with varieties 
$X \subseteq \mathbb{P}^{N-1}$ rather than their affine cones in
$\mathbb{A}^N.$ 

Since defective varieties are
rare, we usually expect equality to hold in sufficiently general
coordinates.   In fact, since a projective curve is never defective,
we will show in Section~\ref{subsec: curves}
that $M(X^{\{s\}})$ is always a Terracini union if $X$ is a projective curve.  In higher dimensions, the question is
especially interesting when coordinates are chosen so that $\II(X)$ is
generated by sparse polynomials.  Motivated by this philosophy, we
investigate this phenomenon for toric varieties and their secant varieties in Section~\ref{subsec: toric}.

\subsection{Curves}\label{subsec: curves}
We show that the algebraic matroids of secant varieties of projective curves are all uniform.

\begin{theorem}
If $X \subseteq \mathbb{P}^{n-1}$ is a nondegenerate irreducible
curve, then $M(X^{\{s\}}) = sM(X)$.  In particular, $M(X^{\{s\}})$
is the uniform matroid of rank $\min\{2s, n\}$ on $Z.$
\end{theorem}
\begin{proof}
We begin by showing that $M(X)$ is the uniform matroid of rank 2 on $Z.$
Let $i \neq j \in Z.$  Suppose for contradiction, that $\{z_i,
z_j\}$ is dependent.  Then there exists a nonzero homogeneous
polynomial $f \in \II(X) \cap \mathbb{C}[z_i,z_j].$  Since
every homogeneous polynomial in two variables factors into linear
factors, and $\II(X)$ is prime, a linear form in $z_i$
and $z_j$ is in $\II(X)$. Since $X$ is nondegenerate, this is a
contradiction, since $X$ cannot be contained in any hyperplane.

By Theorem 10.11 of \cite{eh} as a nondegenerate curve,
$X$ is not $s$-defective. Therefore,
$\dim \widetilde{X^{\{s\}}} = \rank M(X^{\{s\}}) = \min\{2s, N\}.$
If $B$ is a basis of $sM(X)$, since $\overline{\pi_B(X)}$ must be nondegenerate,
Theorem 10.11 of \cite{eh} implies that $\overline{\pi_B(X)}$ is also
non-defective. Theorem~\ref{thm: union equality} then implies that
$M(X^{\{s\}})
= sM(X).$  Since $M(X)$ is the uniform matroid of rank 2, $sM(X)$
is also a uniform matroid.
\end{proof}

\subsection{Toric varieties}\label{subsec: toric}
We present results for projective toric surfaces and 3-folds.  We also discuss our results in 
the context of the results of~\cite{lmr} in Section~\ref{subsec:Laface}.

Interestingly, although defective toric varieties of dimension at least two are rare (see
\cite{cox2007secant}), defective toric projections are quite common.  
Using Theorem~\ref{thm: union equality}, we can show that Example~\ref{ex: veronese} 
generalizes to arbitrary toric surfaces.
\begin{theorem}\label{thm: surface}
Let $X$ be a toric surface corresponding to a lattice polytope $P$.
If $P$ contains a lattice polygon that is a translate
of the convex hull of $\{(0,0), (2,0), (0,2)\},$
then $M^{\{2\}} \neq 2M(X)$.
\end{theorem}
\begin{proof}
Let $U = \{z_i, \ldots, z_{i+5}\}$ correspond to the monomials indicated in
Figure~\ref{fig: veronese}.
\begin{figure}[H]
  \centering
  \includegraphics[width=0.6\linewidth]{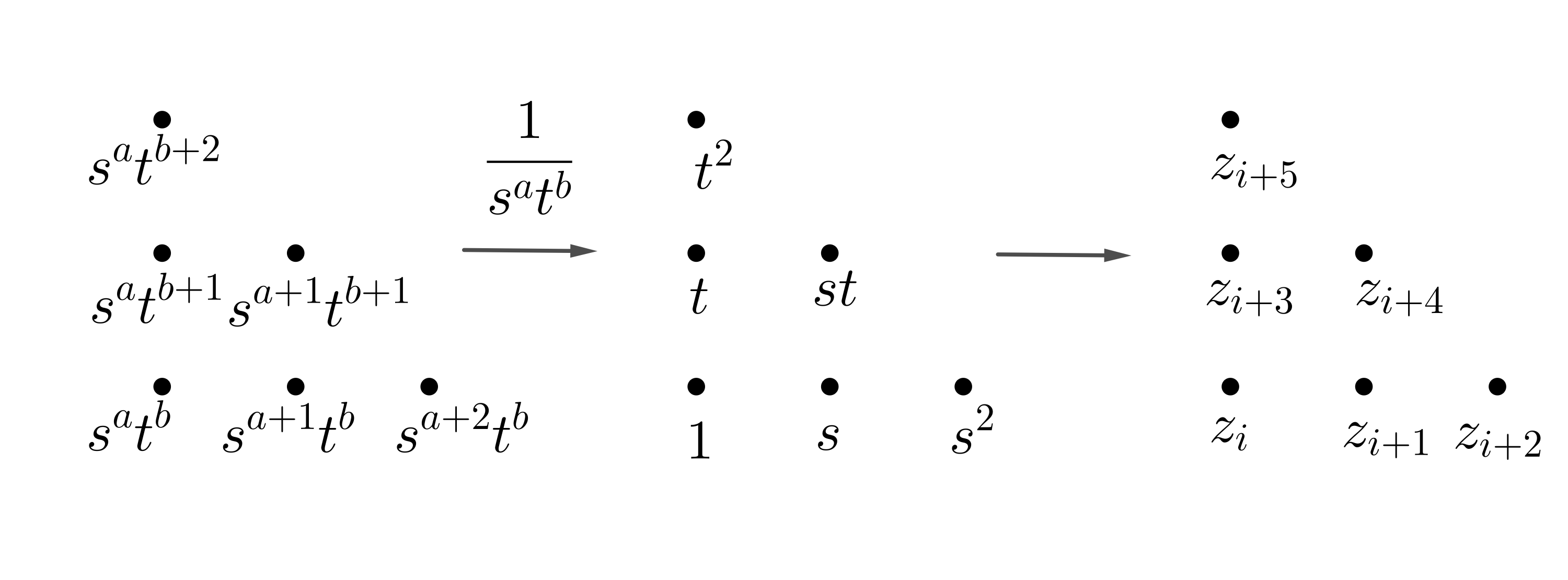}
  \caption{Monomials corresponding to the quadratic Veronese.}
  \label{fig: veronese}
\end{figure}
First, we will show that $U$ is dependent in $M^{\{2\}}.$  Define
\[
  A =
  \begin{bmatrix}
    z_i & z_{i+1} & z_{i+3}\\
    z_{i+1} & z_{i+2}& z_{i+4}\\
    z_{i+3} & z_{i+4} & z_{i+5}
  \end{bmatrix}.
\]
Substituting the corresponding monomials for the $z_i$ into $A$ below
shows that the  $2 \times 2$ minors of $A$ vanish on the torus
embedding and hence are in $\II(X).$
\[
  \begin{bmatrix}
    s^at^b & s^{a+1}t^b & s^at^{b+1}\\
    s^{a+1}t^b & s^{a+2}t^b & s^{a+1}t^{b+1}\\
    s^at^{b+1} & s^{a+1}t^{b+1} & s^at^{b+2}\\
  \end{bmatrix}
  = s^at^b
  \begin{bmatrix}
    1 & s & t\\
    s & s^2 & st\\
    t & st & t^2
  \end{bmatrix},
\]
Moreover, each of the first partials of $\det A$
is in $\II(X),$ so $\det A \in \II(X^{\{2\}})$.  We
conclude that $U$ is dependent in $M^{\{2\}}.$

Now we show that $U$ is independent in $2M(X)$.  Observe that there
are no relations on the monomials $s^at^b\{1,s,t^2\}$ and
$s^at^b\{t, st, s^2\}$.    Therefore, $\{z_i, z_{i+1}, z_{i+5}\}$
and $\{z_{i+2}, z_{i+3}, z_{i+4}\}$ are independent in $M(X).$
Since $U$ is the union of these two disjoint bases, $U$ is
independent in $2M(X).$
\end{proof}

For the Veronese embedding of degree $d$ we can give a lower bound on
how many bases of $2M(X)$ fail to be bases of $M(X^{\{2\}}).$

\begin{corollary}\label{cor: veronese}
Let $d > 2$ and $X$ be the image of $\nu_d:\mathbb{P}^2 \to
\mathbb{P}^{\binom{d+2}2-1}$ via
monomials of degree $d$.  Then $M^{\{2\}} \neq 2M(X)$.  Moreover,
there are at least
$\binom{d}2$ bases of $2M(X)$ that fail to be bases of $M^{\{2\}}$.
\end{corollary}

\begin{proof}
For the second statement, note that the embedding $\nu_d$
corresponds to the lattice polytope $P_d$ with vertices $(0,0),
(d,0),$ and $(0,d).$  For $k = 1, \ldots, d-1$ there are $k$
translates of $P_2$ on the $(d-k)$th level of $P_d.$  (See
Figure~\ref{fig: cubicVeronese} for the case $d=3$.)  Summing, we
have $1+ \cdots + (d-1) = \binom{d}2$ such triangles, each of which
corresponds to a base of $2M(X)$.
\end{proof}

To get a similar result for toric threefolds, we need to produce a
defective toric variety with exactly 8 lattice points that can be
partitioned into two independent sets of four points.
\begin{theorem}\label{thm: threefold}
Let $P$ be the convex hull of $\{(0,0,1), (1,0,2),(0,2,1),(2,2,1),
(1,1,0)\}.$  If $Q \subseteq \mathbb{R}^3$ is a lattice polytope
containing a $PGL(3,\mathbb{Z})$-equivalent translation of $P$,
then $M(X_Q^{\{2\}}) \neq 2M(X_Q).$
\end{theorem}
\begin{proof}
After applying an element of $GL(3,\mathbb{Z})$ and a translation,
we may assume that $Q\subseteq \mathbb{R}^3$ is a lattice polytope
containing $P$.  Remark 6.2 in \cite{lmr} shows that $X_P \subseteq
\mathbb{P}^7$ is 2-defective and that \[Q \cap \mathbb{Z}^3
  \supseteq P \cap \mathbb{Z}^3 = \{(0,0,1), (1,0,2),(0,2,1),(2,2,1),
(1,1,0), (1,1,1), (1,2,1), (0,1,1)\}.\]  Each lattice point in $Q
\cap \mathbb{Z}^3$ corresponds to a coordinate in the ambient
space, and assume that these eight lattice points correspond to the
first eight coordinates.  Therefore, partitioning $[8]$ into two
bases of $M(X_Q)$ corresponds to partitioning the columns of
\[
  \begin{pmatrix}
    1 & 1 & 1 & 1 & 1 & 1 & 1 & 1\\
    0 & 1 & 0 & 2 & 1 & 1 & 1 & 0\\
    0 & 0 & 2 & 2 & 1 & 1 & 2 & 1\\
    1 & 2 & 1 & 1 & 0 & 1 & 1 & 1
  \end{pmatrix}
\]
into two sets of 4 linearly independent columns.  Since the
determinants of the first four and last four columns are nonzero,
we see that $\{1,2,3,4\}$ and $\{5,6,7,8\}$ are bases of $M(X_Q)$
and that the projection of $X_Q$ to the $\mathbb{P}^7$
corresponding to these coordinates is $X_P$, which is defective.
Therefore, by Theorem~\ref{thm: union equality},  $M(X_Q^{\{2\}})
\neq 2M(X_Q).$
\end{proof}

In Theorems~\ref{thm: surface} and~\ref{thm: threefold}, the
Terracini union property failed because of projection to a normal
toric variety.  However, as Example~\ref{ex: nonnormal} shows, the
Terracini union property can also fail because of projections to
non-normal toric varieties.
\begin{example}\label{ex: nonnormal}
Let $X \subseteq \mathbb{P}^{11}$ be the embedding of $\mathbb{P}^1 \times
\mathbb{P}^2$ via a monomial basis for
$H^0(\mathbb{P}^1 \times
\mathbb{P}^2,O_{\mathbb{P}^1 \times \mathbb{P}^2}(1,2)).$  The
exponent vectors of these monomials are the columns of integer matrix
(where we delete the row of 1's)
\[
  \left(
    \begin{array}{cccccccccccc}
      1 & 1 & 1 & 1 & 1 & 1 & 1 & 1 & 1 & 1 & 1 & 1\\
      0 & 0 & 0 & 0 & 0 & 0 & 1 & 1 & 1 & 1 & 1 & 1\\
      0 & 1 & 0 & 2 & 1 & 0 & 0 & 1 & 0 & 2 & 1 & 0\\
      0 & 0 & 1 & 0 & 1 & 2 & 0 & 0 & 1 & 0 & 1 & 2
  \end{array}\right).
\]

The matroid $M(X^{\{2\}})$ fails to be a Terracini union, and one of
the missing bases corresponds to
the projection to the variety given by the submatrix
\[
  \begin{pmatrix}
    1 & 1 & 1 & 1 & 1 & 1 & 1 & 1 \\
    0 & 0 & 0 & 0 & 0 & 1 & 1 & 1 \\
    0 & 0 & 2 & 1 & 0 & 0 & 1 & 0 \\
    0 & 1 & 0 & 1 & 2 & 0 & 0 & 1
  \end{pmatrix}
\]
is defective.  We see that the variety corresponding to the
projection fails to be normal because we have lattice points $(0,0,0)$ and
$(0,2,0)$ but not $(0,1,0)$.
\exqed
\end{example}

\subsubsection{Comparison with tangent-space methods}\label{subsec:Laface}

While our results emphasize projections and matroid unions,
other approaches such as \cite{lmr} use tangent space computations
via Terracini's lemma. To illustrate the contrast, we conclude this
section with a worked example.

The work of \cite{lmr} develops a method to test non-defectiveness of
secant varieties
of toric varieties using Terracini's lemma and tangent space
computations. Their approach
relies on carefully chosen one-parameter subgroups of the torus,
which provide explicit
tangent vectors. Verifying non-defectiveness then reduces to
computing the rank of a matrix, which in turn reduces to finding a
nonzero minor of the appropriate size.  This can be translated into a
combinatorial condition
on lattice points in the defining polytope: one must select disjoint
simplices and choose
subgroups so that certain linear functions are maximized on distinct
simplices. (This is also closely related to the approach taken in
\cite{draisma2008tropical}.) This
ensures that there exists a  unique highest-degree term in the minor
expansion, proving
non-vanishing and hence the desired rank condition.

The following example, we apply Proposition~3.3 of \cite{lmr},
illustrating both their
method and how it compares with the matroid-theoretic perspective
developed here.

\begin{example}[Comparison with \cite{lmr}]\label{ex:Laface}
Let $P$ be a $3\times 2$ lattice rectangle with associated embedding $\varphi_P:\mathbb{P}^1 \times \mathbb{P}^1
\hookrightarrow \mathbb{P}^{11}$ and let $X$ be the image of $\varphi_P$.
Each $v = (v_1,v_2) \in \mathbb{Z}^2$ 
gives rise to a 1-parameter subgroup with elements
$\Gamma_v(a) = (a^{v_1}, a^{v_2})$
for $a \in \mathbb{C}^*$. Following \cite{lmr}, we choose points of the form $\Gamma_v(a) \in (\mathbb{C}^*)^2$ and produce a
$3 \times 12$ matrix whose rows
span the tangent space at $\varphi_P(\Gamma_v(a))$. To study the second
secant variety, we compute two
such matrices and stack them to form a $6\times 12$ matrix.

Proposition~3.3 of \cite{lmr} selects the two simplices below in $P$:
\[
  \begin{tikzpicture}[style = thick, scale = .6]
    \tikzstyle{dot}=[circle,fill=black,inner sep=2 pt];
    \node  [dot] (1) at (0,0) {}; \node  [dot] (2) at (1,0) {};
    \node  [dot] (3) at (2,0) {}; \node  [dot] (4) at (0,1) {};
    \node  [dot] (5) at (1,1) {}; \node  [dot] (6) at (2,1) {};
    \node  [dot] (7) at (0,2) {}; \node  [dot] (8) at (1,2) {};
    \node  [dot] (9) at (2,2) {}; \node  [dot] (10) at (3,0) {};
    \node  [dot] (11) at (3,1) {}; \node  [dot] (12) at (3,2) {};
    \draw (1) --(2) -- (4) -- (1);
    \node at (-1,2) {$\Delta_1$};
    \node at (-1,0) {$\Delta_2$};
    \draw (7) -- (5) -- (8) -- (7);
  \end{tikzpicture}.
\]
The vectors $v_1 = (2,1)$ and $v_2 = (1,1)$ separate $\Delta_1$ and
$\Delta_2$ in
$\Delta = \Delta_1 \cup \Delta_2$. From these we obtain the
one-parameter subgroups
$\Gamma_{(2,1)}(a) = (a^2,a)$ and $\Gamma_{(1,1)}(a) = (a,a)$.
Evaluating the differential
of the torus embedding at $\Gamma_{(2,1)}(2)$ gives
\[
  \left(\!
    \begin{array}{cccccccccccc}
      0&1&8&48&0&2&16&96&0&4&32&192\\
      0&0&0&0&1&4&16&64&4&16&64&256\\
      1&4&16&64&2&8&32&128&4&16&64&256
  \end{array}\!\right),
\]
whose rows form a basis for the tangent space at $\varphi_P(\Gamma_{(2,1)}(2))$. Similarly, the tangent space at
$\varphi_P(\Gamma_{(1,1)}(2))$ is spanned by
\[
  \left(\!
    \begin{array}{cccccccccccc}
      0&1&4&12&0&2&8&24&0&4&16&48\\
      0&0&0&0&1&2&4&8&4&8&16&32\\
      1&2&4&8&2&4&8&16&4&8&16&32
  \end{array}\!\right).
\]

Concatenating these gives a matrix $M$ which has rank $6$.
Terracini's lemma then
implies that the two chosen points are generic enough for their
tangent spaces to span
the tangent space of the secant variety at
$\varphi_P(\Gamma_{(2,1)}(2))+\varphi_P(\Gamma_{(1,1)}(2))$.

However, this choice of $v_1,v_2,a$ is not generic enough to make the
linear matroid on
the rows of $M$ coincide with $M(X^{\{2\}})$. The linear matroid
has $486$ bases, while $M(X^{\{2\}})$ has $916$. Choosing instead $v_1=(5,2)$,
$v_2=(1,1)$, and $a=3$ produces a linear matroid with $916$ bases,
agreeing with $M(X^{\{2\}})$.

In summary, the method of \cite{lmr} computes the rank of $M$ and
hence the dimension
of $X^{\{2\}}$. Since the dimension of a variety equals the rank of
its algebraic matroid,
their method recovers the rank of $M(X^{\{2\}})$. Our framework,
however, also detects when $M(X^{\{2\}})$ fails to be a Terracini
union. For example, because $P$ contains the simplex
below, \[
  \begin{tikzpicture}[style = thick, scale = .6]
    \tikzstyle{dot}=[circle,fill=black,inner sep=2 pt];
    \node  [dot] (1) at (0,0) {}; \node  [dot] (2) at (1,0) {};
    \node  [dot] (3) at (2,0) {}; \node  [dot] (4) at (0,1) {};
    \node  [dot] (5) at (1,1) {}; \node  [dot] (6) at (2,1) {};
    \node  [dot] (7) at (0,2) {}; \node  [dot] (8) at (1,2) {};
    \node  [dot] (9) at (2,2) {}; \node  [dot] (10) at (3,0) {};
    \node  [dot] (11) at (3,1) {}; \node  [dot] (12) at (3,2) {};
    \draw (1) --(3) -- (7) -- (1);
  \end{tikzpicture}
\]
$M(X^{\{2\}})$ is not a Terracini union by Theorem~\ref{thm: union equality}.
\exqed
\end{example}

\section{Open questions}
In this section we ask two questions for further study. The first is
motivated by
Example~\ref{ex: coloop extension}, which showed that $M(X^{\{2\}})$
may be a Terracini union even though $X$ is defective.  However, in
the example given, $X$ is defective because it is a cone.  In terms
of the combinatorics, $M(X)$ and $M(X^{\{2\}})$ have the same coloop,
and $M(X^{\{2\}})$ is a Terracini union.
\begin{question}
Let $X\subseteq \mathbb{A}^n_\kk$ be an irreducible affine cone and
suppose that $X$ is
defective and that $M(X^{\{2\}})$ is a Terracini union.  Must $M(X)$
contain a coloop or loop?
\end{question}
To state the second question we observe that the Terracini union
property suggests a more
general notion of a {\em Terracini matroid}.
\begin{definition}
Let $X \subseteq \mathbb{A}^n$ be a variety and let $M(X)$ be its
algebraic matroid.
We say that $M(X)$ is a \emph{k-fold Terracini matroid} if there exist varieties
$X_1, \ldots, X_k \subseteq \mathbb{A}^N$, with $X_i\neq X$, such that
\begin{equation}\label{eq:terracini}
  X = X_1 + \cdots +X_k
  \quad \text{and} \quad
  M(X) = M(X_1) \,\vee\, \cdots \,\vee\, M(X_k).
\end{equation}
\end{definition}

The second asks about the \emph{rigidity} of a decomposition of an
algebraic matroid as a Terracini
union.
\begin{question}
Let $X\subseteq \mathbb{A}^n_\kk$ be an irreducible affine cone that
is not a linear space, and suppose that $M(X)$ is a $k$-fold
Terracini matroid.  Can $M(X)$ be realized as a Terracini union in
more than one way?
\end{question}

\section*{Acknowledgments} This work originated at the NII Shonan Meeting
“Theory and Algorithms in Graph Rigidity and Algebraic Statistics” in 2024.
We are also grateful to the Isaac Newton Institute for Mathematical Sciences
for hosting us in December 2024 through the INI Mathematical
Retreats program and ICERM (via NSF Grant No. DMS-1929284) for its
hospitality during the semester program on ``Geometry of packings,
materials and rigid
frameworks''. 
We used Macaulay 2 packages \cite{PhylogeneticTreesSource},
\cite{QuasidegreesSource}, \cite{MatroidsSource}, and
\cite{algebraicMatroids} (see the corresponding articles
\cite{PhylogeneticTreesArticle}, \cite{QuasidegreesArticle}, and
\cite{MatroidsArticle})
to compute examples. F.M. was partially supported by the
FWO Odysseus grant G0F5921N and iBOF/23/064 grant from KU Leuven.
L.T. was partially supported by UK Research and Innovation (grant
number UKRI1112),
under the EPSRC Mathematical Sciences Small Grant scheme.
We also wish to thank Bernd Sturmfels, Peter Vermeire and Dario Antolini 
for helpful conversations.

\defbibheading{bibliography}{%
\section*{References}%
\markboth{}{}%
\thispagestyle{plain}%
}

\printbibliography

\clearpage
\markboth{}{}

\appendix

\section{Conormal Terracini}\label{sec: terracini}
Although we do not use it to prove our theorems, the classical Terracini lemma 
has the following dual form.  Like the primal version, it follows directly 
from Sard's theorem, once we have a global coordinate system.
\begin{lemma}\label{lem: terracini}
Let $\kk$ be algebraically closed and of characteristic zero, and let 
$X_1, \ldots, X_s$ be irreducible affine cones in $\kk^n$.  Denote by $X$ the 
embedded join $X_1 + \cdots + X_s$. Then there 
is an open subset $U\subseteq X$ such that, for all 
$x\in U$, if $x = x_1 + \cdots + x_s$, with $x_i\in X_i$, then 
\[
    N^*_x X = \bigcap_{i=1}^s N^*_{x_i} X_i,
\]
where the right-hand side takes the intersection as embedded conormal spaces 
in $(\kk^n)^*$.
\end{lemma}
\begin{proof}
Let $\sigma : X_1 \times \cdots \times X_s \to X$ be the map that adds points.  The set of 
points $(x_1, \ldots, x_s)$ that map to a regular value of $\sigma$  and for which 
each $x_i$ is a regular value of the projection $p_i$ to the $i$th factor is 
open and dense in $X$.  Consider the following diagram:

\begin{center}
\begin{tikzcd}
0 \arrow[r] & N^*_x X \arrow[r] & (\kk^n)^* \arrow[r] \arrow[dr] & T^*_x X \arrow[d, hook, "(d\sigma)^*"] \\
& & & \bigoplus_i T^*_{x_i} X_i
\end{tikzcd}
\end{center}

Because $x\in X$ is smooth, the top row is from the conormal 
exact sequence at $X$.  The injection is from 
Lemma~\ref{lem: product cotangent space} and the fact that $x$ is a 
regular value of $\sigma$.  The diagonal map is the composition, injectivity 
of the right vertical implies that it has the same kernel as the 
surjection to $T^*_x X$, which is $N^*_x X$.  By definition, the 
composite has kernel $\bigcap_{i=1}^s N^*_{x_i} X_i$ (every component 
needs to vanish in the image), so we are done.
\end{proof}

\section{Loops and coloops}\label{sec: loops and coloops}

In  Example~\ref{ex: coloop extension} we saw how coloops can affect
the combinatorics of the algebraic matroid of a join.
In this appendix we give state and prove some basic results, likely
known to experts but not in the literature,
illustrating how both loops and coloops arise in algebraic matroids.
\begin{lemma}\label{lem: geometric coloop}
Let $X\subseteq \mathbb{A}^{N}$ be an irreducible affine cone.  Then
$z_i\in Z$ is a coloop in $M(X)$ if
and only if $X$ is the cone over the point in $\mathbb{A}^{N}$
corresponding to the elementary vector $e_i \in \mathbb{A}^N$.
Algebraically,  $z_i\in Z$ is a coloop in $M(X)$ if
and only if there exist
homogeneous generators $g_1, \ldots, g_t\in S$ for $\II(X)$ that do not contain
$z_i$ in their support.
\end{lemma}
\begin{proof}
We prove the algebraic statement.  Let $P = \II(X) \subseteq S = \kk[Z]$.  Suppose, first,
that none of the $g_j$ are supported on $z_i$.  We
claim that $z_i\notin P$.  Indeed, any linear form in $P$ is a
$\kk$-linear combination of linear forms in the homogeneous
generating set $g_i$.
Since we have assumed none of these are supported on $z_i$, no
linear form with $z_i$ in its support can be in $P$.

Towards a contradiction, we now assume  that $i$ is contained in a
circuit $C$.  Let $f$ be the circuit polynomial of $C$.
Since $f \in P$, $f = f_1g_1+ \cdots +f_tg_t$ for some $f_j \in S.$
Write $f_j = z_iq_j+r_j$ where $r_j$
is not divisible by $z_i$.
Define $h = z_iq_1g_1+\cdots + z_iq_tg_t$ and $h' = r_1g_1+\cdots +
r_tg_t$ so that $f = h + h'$.
By construction, $h'\in P$, and no term in $h'$ is divisible by
$z_i$.  Because
every term of $h$ is divisible by $z_i$, there is no cancellation
between $h$ and $h'$, so the support of $h'$ is properly
contained in that of $f$.  Since $f$ is a circuit polynomial, the
minimality of its support
implies that $h' = 0$, and, hence that $f = h$.  We now have
$f = z_i f'$ for some $f'\in S$.  Because  $z_i\notin P$ and $P$
is prime, $f'\in P$.  Since $f$ is a circuit polynomial,it is
nonzero so $f'$ must be  as well.  We are
now at the desired contradiction: as a circuit polynomial, $f$ is
irreducible, but we have shown that
$f$ is reducible.  We conclude that our
assumption is false; i.e., there is no circuit in $M(X)$ supported on $i$.
Since there is no circuit in $M(X)$ supported on $i$, $i$ is a
coloop in $M(X)$.

Now we assume that $i$ is a coloop in $M(X)$.
Let $S'   =\kk[Z\setminus \{z_i\}]$ and let $P'  = P\cap S'$.
By construction, $P'$ is prime.  The hypothesis that $i$ is a
coloop implies that the bases of $M(P')$
are exactly the bases of $M(X)$ with $i$ removed.  The  dimension
of $P'$ is equal to the rank of its
algebraic matroid, and so we conclude that $\dim S'/P' = \dim S/P -
1$.  The ideal
generated by $P'$ in $S$, $P'S$, is a prime ideal of $S$ contained
in $P$.  Since $P'S$ defines a cone, we see that $\dim S/(P'S) =
\dim S'/P' +1 = \dim S/P.$  Since $P'S \subseteq P$ and both are
prime ideals defining varieties of the same dimension, they must be
equal. Therefore, we see that $P$ can be generated by elements not
containing $z_i$ in their support.
\end{proof}
Loops are a bit easier.
\begin{lemma}\label{lem: geometric loops}
Let $X\subseteq \mathbb{A}^{N}$ be an irreducible affine cone.  Then
$i$ is a loop in $M(X)$ if and only if $X$ is contained in the
hyperplane $\mathcal{V}(z_i)$.  Algebraically, $i$ is a loop in
$M(X)$ if and only if $z_i \in \II(X)$.
\end{lemma}
\begin{proof}
The variety $X$ is contained in $\mathcal{V}(z_i)$ if and only
if $z_i\in \II(X)$.  We claim that this latter statement is equivalent
to $i$ being a loop in $M(X)$.

Suppose that $i$ is a loop in $M(X)$.  The circuit polynomial of
the circuit $\{i\}$ is
homogeneous, as it is in $\II(X)$ and is supported only on $z_i$, so it
must be of the form $z_i^n$.
Since a circuit polynomial is irreducible, it must be $z_i$.  Conversely, if
$z_i\in \II(X)$, then it is a circuit polynomial for $\{i\}$ in
$M(X)$, since it is irreducible, has minimal support,
and is supported only on $z_i$.
\end{proof}
We pause to record a fact we do not need, but might be interesting.
While the matroid $M(X^*)$ of the
dual variety of a projective variety is not necessarily the dual
matroid  $M(X)^*$.  However, loops and coloops
do exchange under projective duality on $X$.
\begin{lemma}\label{lem: dual variety loops coloops}
Let $X\subseteq \mathbb{A}^{N}$ be an irreducible homogeneous variety.  Then
$i\in [N]$ is a coloop in $M(X)$ if and only if
$i$ is a loop in $M(X^*)$.
\end{lemma}
\begin{proof}
It is shown in \cite[Theorem 5.3]{fulton2013intersection} that an
irreducible projective variety $X$ is a cone over a point $p$
if and only if its dual variety $X^*$ is contained in the
hyperplane projectively dual to $p$.  Applying Lemmas~\ref{lem: geometric coloop} and~\ref{lem: geometric loops}
completes the proof.
\end{proof}
Now we can see that an affine cone $X$ with a coloop in its matroid
is defective whenever its second secant does not fill $\mathbb{A}^{N}$.
\begin{lemma}\label{lem: coloop defective}
Let $X\subseteq \mathbb{A}^{N}$ be an irreducible homogeneous variety.  If
$M(X)$ has a coloop, then either $X^{\{2\}}= \mathbb{A}^{N}$, or
$X$ is defective.
\end{lemma}
\begin{proof}
Suppose that $X$ has dimension $r$ and that $z_i$ is a coloop in
$M(X)$.  If the expected dimension of $X^{\{2\}}$ is
$N$, then $X$ is defective iff $X^{\{2\}}\neq \mathbb{A}^{N}$,
so we are done.  Hence, we may assume from now on that the
expected dimension is $2r$.

By Lemma~\ref{lem: geometric coloop}, $X$ is a cone with vertex $e_i$ over an
irreducible variety $X'\subseteq \mathbb{A}^{N-1}$ of dimension
$r-1$.  Hence, $X^{\{2\}}$ is a cone over $(X')^{\{2\}}$ with
vertex $e_i$.  The dimension of
$(X')^{\{2\}}$ is at most $2r - 2$, so $\dim X^{\{2\}} \leq 2r - 1
< 2r$, which shows that $X$
is defective.
\end{proof}
\end{document}